\documentclass{amsart}
\usepackage{amsfonts,graphicx,rotating}
\usepackage{amssymb}
\usepackage[hidelinks]{hyperref}
\usepackage{tikz-cd}
\usepackage{ctable}

\makeatletter
\@namedef{subjclassname@2020}{\textup{2020} Mathematics Subject Classification}
\makeatother

\newcommand{\im}{\operatorname{im}}
\newcommand{\coker}{\operatorname{coker}}

\newcommand{\gr}{\widetilde G}
\newcommand{\wg}{\widetilde \gamma}
\newcommand{\w}{\widetilde w}
\newcommand{\wa}{\widetilde a}
\newcommand{\ww}{\mathrm w}

\newtheorem{theorem}{Theorem}[section]
\newtheorem{lemma}[theorem]{Lemma}
\newtheorem{corollary}[theorem]{Corollary}
\newtheorem{proposition}[theorem]{Proposition}

\theoremstyle{definition}
\newtheorem{definition}[theorem]{Definition}

\theoremstyle{remark}
\newtheorem{remark}[theorem]{Remark}

\numberwithin{equation}{section}


\begin{document}

\title[Cohomology rings of oriented Grassmann manifolds $\widetilde G_{2^t,4}$]
{Cohomology rings of oriented Grassmann manifolds $\widetilde G_{2^t,4}$}

\author{Uro\v s A.\ Colovi\' c}
\address{}
\email{urosandrijacolovic@gmail.com}

\author{Milica Jovanovi\'c}
\address{University of Belgrade,
  Faculty of mathematics,
  Studentski trg 16,
  Belgrade,
  Serbia}
\email{milica.jovanovic@matf.bg.ac.rs}

\author{Branislav I.\ Prvulovi\'c}
\address{University of Belgrade,
  Faculty of mathematics,
  Studentski trg 16,
  Belgrade,
  Serbia}
\email{branislav.prvulovic@matf.bg.ac.rs}

\thanks{The second author was partially supported by the Science Fund of the Republic of Serbia, Grant No.\ 7749891, Graphical Languages -- GWORDS} 
\thanks{The third author was partially supported by the Ministry of Science, Technological Development and Innovations of the Republic of Serbia [contract no.\ 451-03-47/2023-01/200104].}





\begin{abstract}
We give a description of the mod 2 cohomology algebra of the oriented Grassmann manifold $\widetilde G_{2^t,4}$ as the quotient of a polynomial algebra by a certain ideal. In the process we find a Gr\"obner basis for that ideal, which we then use to exhibit an additive basis for $H^*(\widetilde G_{2^t,4};\mathbb Z_2)$.
\end{abstract}

\maketitle



\section{Introduction and statement of main results}
\label{intro}

The Grassmann manifold $G_{n,k}$ is defined to be the manifold of all $k$-dimensional subspaces of $\mathbb R^n$ and the oriented Grassmann manifold $\widetilde G_{n,k}$ is the manifold of all oriented $k$-subspaces of $\mathbb R^n$. There are homeomorphisms $G_{n,k} \approx G_{n,n-k}$ and $\widetilde G_{n,k} \approx \widetilde G_{n,n-k}$, so there is no loss of generality in assuming $n \geqslant 2k$. The mod 2 cohomology algebra of $G_{n,k}$ has been determined by Borel in \cite{Borel} to be a quotient of the polynomial algebra over Stiefel--Whitney classes of the canonical vector bundle $\gamma_{n,k}$. Although there is a (universal) double cover $p:\widetilde G_{n,k}\rightarrow G_{n,k}$, far less is known about the mod 2 cohomology of the oriented Grassmann manifold $\widetilde G_{n,k}$. For $k = 2$, Korba\v s and Rusin gave a complete description of $H^*(\widetilde G_{n,2}; \mathbb Z_2)$ in \cite{KorbasRusin:Palermo}. Basu and Chakraborty made first steps for $k = 3$ in \cite{BasuChakraborty}, which were improved first in \cite{CP} for $n$ a power of two, and then in \cite{JP} for $n\in\{2^t-3,2^t-2,2^t-1\}$, $t\geqslant3$ (for $n = 2^t-1$ there is still an ambiguity left). The remaining values of $n$ were examined by Matszangosz and Wendt in \cite{Wendt}, where the authors gave a complete description of $H^*(\widetilde G_{n,3};\mathbb Z_2)$ for $2^{t-1}< n \leqslant 2^t-4$.

In this paper we go a step further and consider the case $k = 4$ and $n$ a power of two. The fact that among all Grassmannians $\widetilde G_{n,3}$ those with $n=2^t$ have the simplest cohomology algebra structure led us to begin examining the case $k=4$ starting with $n$ a power of two. Apart from commonly known facts (reviewed in Section \ref{sec:preliminaries}) there have been very few results so far concerning $H^*(\widetilde G_{n,4};\mathbb Z_2)$. The characteristic rank of $H^*(\widetilde G_{n,4};\mathbb Z_2)$ was studied and computed in \cite{Korbas:Osaka} and \cite{bane-marko}. In \cite{Rusin} Rusin determined generators of the cohomology groups $H^j(\widetilde G_{n,4};\mathbb Z_2)$ for $n = 8, 9, 10, 11$. In \cite{Wendt} Matszangosz and Wendt performed some computer calculations to exhibit the degrees of the so-called anomalous classes in $H^*(\widetilde G_{n,4};\mathbb Z_2)$ for $n\leqslant32$.

Our main result is contained in the following statement.

\begin{theorem}\label{Cohomology_G_2^t,4} Let $t\geqslant3$. There is an isomorphism of graded algebras
    \begin{equation}\label{main_thm}
        H^*(\widetilde G_{2^t,4};\mathbb Z_2) \cong \frac{\mathbb Z_2[ w_2, w_3, w_4, a_{2^t-4}]}{(g_{2^t-2}, g_{2^t-1}, g_{2^t}, a_{2^t-4}^2+Pa_{2^t-4}+Q)},\end{equation}
    where $|w_i| = i$, $i=2,3,4$, $|a_{2^t-4}|=2^t-4$, and $P$ and $Q$ are some polynomials in $w_2$, $w_3$ and $w_4$.
\end{theorem}

The polynomials $g_r$ appearing in (\ref{main_thm}) are well-known polynomials in $w_2$, $w_3$ and $w_4$. 
The cosets of the variables $w_2$, $w_3$ and $w_4$ correspond via the isomorphism (\ref{main_thm}) to the Stiefel--Whitney classes $\widetilde w_2, \widetilde w_3$ and $\widetilde w_4$ (respectively) of the canonical vector bundle $\widetilde \gamma_{2^t,4}$ over $\widetilde G_{2^t,4}$. The coset of the fourth variable $a_{2^t-4}$ corresponds to a class $\widetilde a_{2^t-4}\in H^{2^t-4}(\widetilde G_{2^t,4};\mathbb Z_2)$, which is not in the image of the induced map $p^*:H^*(G_{2^t,4};\mathbb Z_2)\rightarrow H^*(\widetilde G_{2^t,4};\mathbb Z_2)$ (equivalently, $\widetilde a_{2^t-4}$ is not expressible as a polynomial in $\widetilde w_2, \widetilde w_3$ and $\widetilde w_4$).


We prove Theorem \ref{Cohomology_G_2^t,4} in two steps. First, we show that, besides Stiefel--Whitney classes $\widetilde w_2, \widetilde w_3$ and $\widetilde w_4$, one more cohomology class $\widetilde a_{2^t-4}\in H^{2^t-4}(\widetilde G_{2^t,4};\mathbb Z_2)$ suffices to generate $H^*(\widetilde G_{2^t,4};\mathbb Z_2)$ as an algebra. We do this by applying the method of Basu and Chakraborty \cite{BasuChakraborty} -- we examine the Leray--Serre spectral sequence of two different sphere bundles. Next, we show that $H^*(\widetilde G_{2^t,4};\mathbb Z_2)$ has the multiplicative structure as stated in (\ref{main_thm}). The main tools used in this part include the theory of Gr\"obner bases and the Leray--Hirsch theorem, which, under certain conditions, gives an additive isomorphism between the cohomology of the total space of a fiber bundle and the cohomology of the product of its fiber and base space. We use this theorem in considering the square $\widetilde a_{2^t-4}^2$. Since $H^*(\widetilde G_{2^t,2};\mathbb Z_2) \cong\im p^*\otimes\Lambda(\wa_{2^t-2})$ \cite[Theorem 2.1(b)]{KorbasRusin:Palermo} and $H^*(\widetilde G_{2^t,3};\mathbb Z_2) \cong\im p^*\otimes\Lambda(\wa_{2^t-1})$ \cite[Theorem 1.1]{CP}, one might expect that $H^*(\widetilde G_{2^t,4};\mathbb Z_2)$ also splits as the tensor product of $\im p^*$ and the exterior algebra on $\widetilde a_{2^t-4}$, i.e., that $\widetilde a_{2^t-4}^2=0$. However, this is not the case, as we show in Proposition \ref{Pneq0}. By using the methods of Gr\"obner bases, in Corollary \ref{cor:aditivna_baza} we obtain an additive basis for $H^*(\widetilde G_{2^t,4};\mathbb Z_2)$. 

The paper is organized as follows. In Section \ref{sec:preliminaries} we give the necessary background on the cohomology of Grassmann manifolds as well as the theory of Gr\" obner bases. In Section \ref{sec:Indecomposables} we determine the indecomposable elements (algebra generators) of $H^*(\gr_{2^t,4};\mathbb Z_2)$, and thus establish that $H^*(\widetilde G_{2^t,4};\mathbb Z_2)$ is isomorphic to the quotient of the polynomial algebra $\mathbb Z_2[ w_2, w_3, w_4, a_{2^t-4}]$ by some ideal. A Gr\"obner basis for this ideal is computed in Section \ref{sec:main_proof}, which is devoted to proving Theorem \ref{Cohomology_G_2^t,4}. The Gr\"obner basis is then used to present an additive basis for $H^*(\widetilde G_{2^t,4};\mathbb Z_2)$. In this section we also establish that $\wa_{2^t-4}^2\neq0$ in $H^*(\gr_{2^t,4};\mathbb Z_2)$; more precisely, that the polynomial $P$ from (\ref{main_thm}) cannot be zero.  In the last Section \ref{sec:example} we illustrate the usage of our additive basis by considering the particular case $t=3$ and deducing some additional information on $H^*(\gr_{8,4};\mathbb Z_2)$.

\section{Preliminaries}
\label{sec:preliminaries}

\subsection{Cohomology of Grassmann manifolds}

We will be working solely with the mod 2 cohomology, so we write $H^*(-)$ for $H^*(-;\mathbb Z_2)$.

Due to Borel \cite{Borel} it is well known that for the mod 2 cohomology ring of the Grassmann manifold the following holds:
\begin{equation}\label{Borel}
H^*(G_{n,k})\cong \frac{\mathbb Z_2[w_1,w_2,\dots,w_k]}{(\overline w_{n-k+1},\dots,\overline w_n)},
\end{equation}
where $w_i$ for $1\leqslant i \leqslant k$ corresponds to the $i$-th Stiefel--Whitney class  $w_i(\gamma_{n,k})$ of the canonical vector bundle $\gamma_{n,k}$ over $G_{n,k}$, and $\overline w_r$ for $r\geqslant0$ are polynomials in $w_1,w_2,\dots,w_k$ given by the relation
\begin{equation}\label{eq_mult}
(1+w_1+w_2+\cdots + w_k)(\overline w_0 + \overline w_1 +\overline w_2+\cdots) = 1.\end{equation}
In other words, $\overline w_0+\overline w_1 + \overline w_2+ \cdots$ is the inverse of the element $1+w_1+w_2+\cdots+w_k$ in the ring of power series $\mathbb Z_2[[w_1,w_2,\dots,w_k]]$. The (cosets of the) polynomials $\overline w_r$, $r\geqslant0$, correspond via the above isomorphism to the Stiefel--Whitney classes of the orthogonal complement $\gamma_{n,k}^\perp$ of the canonical bundle $\gamma_{n,k}$.

From (\ref{eq_mult}) one gets the following recurrent formula:
\begin{equation}\label{eq_rec_w}
    \overline w_r = w_1\overline w_{r-1} + w_2 \overline w_{r-2} + \cdots + w_k \overline w_{r-k}, \quad r\geqslant k.
\end{equation}
Moreover, it is not difficult to deduce the following equation from (\ref{eq_mult}):
\begin{equation}\label{eq_explicitly_w}
    \overline w_r = \sum_{a_1+2a_2+\cdots +ka_k = r} [a_1,a_2,\dots, a_k]\,  w_1^{a_1} w_2^{a_2}\cdots w_k^{a_k},\quad r\geqslant0,
\end{equation}
where 
\[[a_1,a_2,\dots, a_k] = \binom{a_1+a_2+\cdots + a_k}{a_1}\binom{a_2+\cdots +a_k}{a_2}\cdots\binom{a_{k-1} + a_k}{a_{k-1}}.\]
Recall that all of the previous equations are in a $\mathbb Z_2$-algebra, so these binomial coefficients are either 0 or 1.

The Grassmann manifold $G_{n,k}$ is covered by the oriented Grassmann manifold $\widetilde G_{n,k}$ via the double cover $p: \widetilde G_{n,k} \rightarrow G_{n,k}$ that sends each oriented $k$-subspace in $\mathbb R^n$ to the same $k$-subspace, but without the orientation. The induced homomorphism $p^*:H^*(G_{n,k})\rightarrow H^*(\widetilde G_{n,k})$ sends each class $w_i(\gamma_{n,k})$, $1\leqslant i \leqslant k$, to the corresponding Stiefel--Whitney class $\widetilde w_i = w_i(\widetilde \gamma_{n,k})$ of the oriented canonical bundle $\widetilde\gamma_{n,k}$ over $\widetilde G_{n,k}$. The image of $p^*$ is a subalgebra of $H^*(\widetilde G_{n,k})$, and it can be shown that 
\begin{equation}
\label{imp*}
    \im p^* \cong \frac{\mathbb Z_2[w_2,\dots,w_k]}{(g_{n-k+1},\dots,g_n)}
    \end{equation}
(see e.g.\ \cite[p.\ 197]{CP}), where $w_i$ corresponds to $\widetilde w_i$ and $g_r\in\mathbb Z_2[w_2,\dots,w_k]$ is the modulo $w_1$ reduction of the polynomial $\overline w_r\in\mathbb Z_2[w_1,w_2,\dots,w_k]$, for $r\geqslant0$. Therefore, reducing equations (\ref{eq_rec_w}) and (\ref{eq_explicitly_w}) modulo $w_1$ one obtains:
\begin{equation}\label{eq_rec_g}
    g_r =   w_2 g_{r-2} + \cdots +   w_k g_{r-k}, \quad r\geqslant k,
\end{equation}
and
\begin{equation}\label{eq_explicitly_g}
    g_r = \sum_{2a_2+\cdots +ka_k = r} \binom{a_2+\cdots +a_k}{a_2}\cdots\binom{a_{k-1} + a_k}{a_{k-1}}\,  w_2^{a_2}\cdots w_k^{a_k},\quad r\geqslant0.
\end{equation}
Actually, if we define $g_{-1}=g_{-2}=\cdots=g_{-k+1}:=0$, then it is routine to check that (\ref{eq_rec_g}) holds for all integers $r\geqslant1$. Furthermore, using induction on $i$, this relation is easily generalized to
\begin{equation}\label{eq_rec_g_gen}
    g_r =   w_2^{2^i} g_{r-2\cdot 2^i} + \cdots +   w_k^{2^i} g_{r-k\cdot 2^i}=\sum_{j=2}^kw_j^{2^i} g_{r-j\cdot 2^i}, \quad r\geqslant1+k\cdot(2^i-1),\, i\geqslant0.
\end{equation}
Namely, the base case $i=0$ is (\ref{eq_rec_g}) (for $r\geqslant1$); and if $i\geqslant1$, assuming that (\ref{eq_rec_g_gen}) is true for $i-1$, then for $r\geqslant1+k\cdot(2^i-1)$ we have
\[g_r =\sum_{j=2}^kw_j^{2^{i-1}} g_{r-j\cdot 2^{i-1}}=\sum_{j=2}^k\sum_{l=2}^kw_j^{2^{i-1}}w_l^{2^{i-1}} g_{r-j\cdot 2^{i-1}-l\cdot2^{i-1}}=\sum_{j=2}^kw_j^{2^i}g_{r-j\cdot 2^i},\]
since the summands with $l\neq j$ cancel out (only those for $l=j$ remain).

It is a known fact that the vector space dimension of $\im p^*$ is half of the dimension of $H^*(\widetilde G_{n,k})$. In the next lemma we prove that this dimension equation holds more generally -- for any sphere bundle $p$. This fact will play an important role in proving Theorem \ref{Cohomology_G_2^t,4}.

\begin{lemma}\label{1/2dim}
    Let $n$ be a positive integer and $S^{n-1}\rightarrow E \xrightarrow{p} B$ a sphere bundle such that $H^*(B)$ is a finite-dimensional vector space (over $\mathbb Z_2$). If $p^*:H^*(B)\rightarrow H^*(E)$ is the induced morphism, then
    \[\dim H^*(E)  = 2\dim(\im p^*).\]
\end{lemma}
\begin{proof} The Gysin sequence for the sphere bundle $S^{n-1}\rightarrow E \xrightarrow{p} B$ is given in the following diagram
\[\cdots \rightarrow H^{i+n-1}(B) \xrightarrow{p^*} H^{i+n-1}(E) \rightarrow H^i(B) \xrightarrow{\cdot e} H^{i+n}(B) \xrightarrow{p^*} H^{i+n}(E)\rightarrow \cdots\]
where $e\in H^{n}(B)$ is the mod 2 Euler class of the bundle. This long exact sequence can be broken up into short exact sequences:
\[0\rightarrow C_i \rightarrow H^i(B) \rightarrow K_i \rightarrow 0, \quad i\in\mathbb Z,\]
where 
\[C_i = \coker\left(p^*:H^{i+n-1}(B)\rightarrow H^{i+n-1}(E)\right)\] and
\[K_i = \ker\left(p^*:H^{i+n}(B)\rightarrow H^{i+n}(E)\right).\] 
Taking the direct sum of these sequences leads to the short exact sequence
\[0\rightarrow \coker p^* \rightarrow H^*(B) \rightarrow \ker p^* \rightarrow 0.\]
Since $H^*(B)$ is finite-dimensional, we have 
\[\dim H^*(B)=\dim(\coker p^*)+\dim(\ker p^*)=\dim H^*(E)-\dim(\im p^*)+\dim(\ker p^*).\]
Finally, the well-known equality 
\[\dim H^*(B)=\dim(\ker p^*)+\dim(\im p^*)\] 
finishes the proof.
\end{proof}



\subsection{Gr\" obner bases}

Here we will give a short introduction to 
        Gr\"obner bases, but for the sake of concreteness we will work over the field $\mathbb Z_2$. For more details see \cite{Becker}.
        
        If $\mathbb Z_2[\overline X]$ is the polynomial ring over one or more variables, than the set of all monomials in $\mathbb Z_2[\overline X]$ will be denoted by $M$. Let $\preceq$ be a monomial ordering on $M$, that is, a well ordering on $M$ with the property that for all $m_1,m_2,m_3\in M$, $m_1\preceq m_2$ implies $m_1m_3\preceq m_2m_3$. For $f = \sum_{i=1}^r m_i\in\mathbb Z_2[\overline X]\setminus\{0\}$, where $m_i$ are distinct monomials,
        define the leading monomial $\mathrm{LM}(f)$ as $\max \{m_i\mid 1\leqslant i \leqslant r\}$, where the maximum is taken with respect to $\preceq$.       
        There are many equivalent ways to define Gr\" obner basis for an ideal (see \cite[Theorem 5.35, Proposition 5.38]{Becker}) and here we give one of them.
        
\begin{definition}
    Let $G \subset\mathbb Z_2[\overline X]\setminus\{0\}$ be a finite set of polynomials, and an $I \unlhd \mathbb\mathbb Z_2[\overline X]$ ideal such that $G\subset I$.
    We say $G$ is a {\it Gr\" obner basis} for $I$ w.r.t.\ $\preceq$
    if the set of cosets of monomials not divisible by any $\mathrm{LM}(f)$ for $f \in G$, comprises a vector space basis for the quotient $\mathbb Z_2[\overline X]/I$.
\end{definition}
    For  $f,g \in\mathbb Z_2[\overline X]\setminus\{0\}$, the $S$-polynomial is defined as
    $$
    S(f,g)=\frac{\mathrm{lcm}\left(\mathrm{LM}(f),\mathrm{LM}(g) \right)}{\mathrm{LM}(f)}\cdot f+
        \frac{\mathrm{lcm}\left(\mathrm{LM}(f),\mathrm{LM}(g) \right)}{\mathrm{LM}(g)}\cdot g,
    $$
    where $\mathrm{lcm}\left(\mathrm{LM}(f),\mathrm{LM}(g) \right)$ is the least common multiple of the leading monomials of $f$ and $g$. Note that $S(f,f)=0$ and $S(g,f)=S(f,g)$.
    
    We are also going to need the notion of reduction of polynomials. Let $G \subset\mathbb Z_2[\overline X]$ be a finite set of polynomials. We say that polynomial $p$ reduces to $q$ modulo $G$ if there exist $n\geqslant  1$ and polynomials $r_1,\ldots,r_n$ such that $r_1=p$, $r_n=q$ and for each 
    $i \in \{1,\ldots , n-1\}$ we have 
    \begin{equation*}
        r_{i+1}=r_i+m_if_i,
    \end{equation*}
    for some $f_i \in G$ and some monomials $m_i \in M$  such that $m_i\mathrm{LM}(f_i)=\mathrm{LM}(r_i)$, $1\leqslant i\leqslant n-1$.
    
    The following theorem gives us an equivalent condition for a set of polynomials to be a Gr\"obner basis,
    and this is what we are going to exploit in Section \ref{sec:main_proof}. It is also known as the Buchberger criterion (\cite[Theorem 5.48]{Becker}).
    \begin{theorem}\label{buchberger}
        Let $G \subset\mathbb Z_2[\overline X]$ be a finite set of nonzero polynomials, and let $I$ be the ideal generated by $G$. Then $G$ is a Gr\" obner basis for $I$ if and only if $S(f, g)$ reduces to zero modulo $G$ for every $f,g \in G$.
        \end{theorem}

\section{A generating set for the algebra $H^*(\widetilde G_{2^t,4})$}
\label{sec:Indecomposables}

In what follows, $t\geqslant3$ is a fixed integer. Following Basu and Chakraborty \cite{BasuChakraborty} we consider two sphere bundles with the same total space. The first one is the sphere bundle associated to the orthogonal complement of the canonical vector bundle $\wg_{2^t,3}$ over $\gr_{2^t,3}$. So the total space is
\[W_{3,1}^{2^t} = \{(P,v)\in\gr_{2^t,3}\times S^{2^t-1} \mid v\perp P\},\]
and the bundle is $S^{2^t-4}\xrightarrow{i} W_{3,1}^{2^t} \xrightarrow{p_1} \gr_{2^t,3}$, where $p_1$ is the projection onto the first coordinate.

The second one is $S^3\rightarrow W_{3,1}^{2^t} \xrightarrow{sp} \gr_{2^t,4}$, where $sp$ sends a pair of an oriented 3-dimensional subspace in $\mathbb R^{2^t}$ and a vector orthogonal to it to the oriented 4-dimensional subspace they span.




\subsection{Generators for $H^*(W_{3,1}^{2^t})$} 

The $E_2$ page of the Leray--Serre spectral sequence associated to the sphere bundle $S^{2^t-4}\xrightarrow{i}W_{3,1}^{2^t}\xrightarrow{p_1} \gr_{2^t,3}$  is concentrated in two rows, and the only differential that can be nontrivial is $d_{2^t-3}$. We know that $d_{2^t-3}$ is the multiplication with the mod $2$ Euler class of the bundle $\wg_{2^t,3}^\perp$, which is equal to the top Stiefel--Whitney class $w_{2^t-3}(\wg_{2^t,3}^\perp)$ of this bundle. On the other hand, this class corresponds to (the coset of) $g_{2^t-3}$ via (\ref{imp*}) for $k=3$. It is a well-known fact that $g_{2^t-3}=0$ (see \cite[Lemma 2.3(i)]{Korbas:Osaka}). We conclude that $w_{2^t-3}(\wg_{2^t,3}^\perp)=0$, and so the spectral sequence collapses at the $E_2$-page. Thus, 
\[E_\infty=E_2\cong H^*(\gr_{2^t,3})\otimes H^*(S^{2^t-4}).\]
Therefore, since the spectral sequence converges to $H^*(W_{3,1}^{2^t})$, we can choose a class $c_{2^t-4}\in H^{2^t-4}(W_{3,1}^{2^t})$ with the property $i^*(c_{2^t-4})=s$, where $s\in H^{2^t-4}(S^{2^t-4})$ is the generator. Also, by \cite[Theorem 1.1]{CP}, $H^*(\gr_{2^t,3})\cong\im p^*\otimes\Lambda(\wa_{2^t-1})$, leading to the conclusion that $E_\infty$ is generated (as a bigraded algebra) by $\w_2\otimes 1$, $\w_3\otimes 1$, $\wa_{2^t-1}\otimes 1$ and $1\otimes s$. This means that, if we define 
\begin{equation}\label{defw}
  \ww_2:=p_1^*(\w_2),\, \ww_3:=p_1^*(\w_3)\, \mbox{ and }\, \mathrm a_{2^t-1}:=p_1^*(\wa_{2^t-1}),  
\end{equation} 
then $H^*(W_{3,1}^{2^t})$ is generated (as a graded algebra) by $\ww_2$, $\ww_3$, $\mathrm{a}_{2^t-1}$ and $c_{2^t-4}$ (cf.\ \cite[Example 1.K]{McCleary}). Hence, we have proved the following proposition.





\begin{proposition}\label{cohW}
    The previously defined classes $\ww_2\in H^2(W_{3,1}^{2^t})$, $\ww_3\in H^3(W_{3,1}^{2^t})$, $c_{2^t-4}\in H^{2^t-4}(W_{3,1}^{2^t})$ and $\mathrm{a}_{2^t-1}\in H^{2^t-1}(W_{3,1}^{2^t})$ constitute a set of algebra generators for $H^*(W_{3,1}^{2^t})$.
\end{proposition}

Since $\wa_{2^t-1}^2=0$ in $H^*(\gr_{2^t,3})$, note that we have
\begin{equation}\label{a^2=0}
    \mathrm{a}_{2^t-1}^2 = p_1^*(\wa_{2^t-1}^2) = 0.
\end{equation}


\subsection{Generators for $H^*(\gr_{2^t,4})$}

Now we move onto determining a generating set for $H^*(\gr_{2^t,4})$. Consider the second sphere bundle $S^3\rightarrow W_{3,1}^{2^t}\xrightarrow{sp} \widetilde G_{2^t,4}$. First, let us note that the total space $W_{3,1}^{2^t}$ is a closed (simply) connected manifold of dimension $2^{t+2}-13$, because both the base space and the fiber are closed (simply) connected manifolds of dimensions $4\cdot2^t-16$ and $3$ respectively. Therefore, the Poincar\'e duality applies to $H^*(W_{3,1}^{2^t})$.

The Leray--Serre spectral sequence associated to the bundle converges to the graded algebra $H^*(W_{3,1}^{2^t})$. Therefore, there exists a stable filtration $\{F_s^r\}_{r,s\geqslant0}$ of $H^*(W_{3,1}^{2^t})$, 
\[H^r(W_{3,1}^{2^t})=F_0^r\supseteq F_1^r\supseteq\cdots\supseteq F_{r-1}^r\supseteq F_r^r\supseteq F_{r+1}^r=F_{r+2}^r=\cdots=0 \,\mbox{ for each } r,\]
such that the induced bigraded algebra $\displaystyle\bigoplus_{p,q\geqslant0}F_p^{p+q}/F_{p+1}^{p+q}$ is isomorphic to the bigraded algebra $E_\infty=\displaystyle\bigoplus_{p,q\geqslant0}E_\infty^{p,q}$ coming from the spectral sequence. 

The total space $W_{3,1}^{2^t}$ is easily identified with the space $\{(V,v)\in\gr_{2^t,4}\times S^{2^t-1}\mid v\in V\}$ (namely, the point $(P,v)\in W_{3,1}^{2^t}$ corresponds to the point $(sp(P,v),v)$ of the latter space), and this is the space of unit vectors in the total space of the canonical bundle $\wg_{2^t,4}$. Therefore, the bundle $S^3\rightarrow W_{3,1}^{2^t}\xrightarrow{sp} \widetilde G_{2^t,4}$ is actually the sphere bundle associated to $\wg_{2^t,4}$. We conclude that the Euler class of this bundle is $w_4(\wg_{2^t,4})=\w_4$, which means that the only (possibly) nontrivial differential $d_4$ in the spectral sequence comes down to the multiplication by $\w_4$. 
This means that
\begin{equation}\label{kersp^*}
 \ker sp^*=\w_4 H^*(\gr_{2^t,4}).   
\end{equation}
In other words, a class in $H^r(\gr_{2^t,4})\cong H^r(\gr_{2^t,4})\otimes H^0(S^3)\cong E_2^{r,0}$ survives until $E_\infty^{r,0}$ if and only if it is not divisible by $\w_4$. Also, the following lemma essentially states that $E_\infty^{r,3}=0$ for $r<2^t-4$ (see Figure \ref{fig:sn}). As a consequence, we have
\begin{equation}\label{F_r^r}
 H^r(W_{3,1}^{2^t})=F_0^r=F_r^r \, \mbox{ for } r<2^t-1.  
\end{equation}

\begin{lemma}\label{d_4mono}
 The multiplication by $\w_4$, $H^r(\gr_{2^t,4})\xrightarrow{\cdot\w_4}H^{r+4}(\gr_{2^t,4})$, is a mono\-morphism if $r<2^t-4$.
\end{lemma}
\begin{proof}
    By \cite[Theorem 2.1(2)]{Korbas:Osaka}, for $r<2^t-4$ every class in $H^r(\gr_{2^t,4})$ is a polynomial in Stiefel--Whitney classes $\w_2$, $\w_3$ and $\w_4$. On the other hand, according to (\ref{imp*}) the subalgebra $\im p^*$ (generated by these three Stiefel--Whitney classes) is isomorphic to the quotient $\mathbb Z_2[w_2,w_3,w_4]/(g_{2^t-2},g_{2^t-1},g_{2^t})$ ($g_{2^t-3}=0$ by \cite[Lemma 2.3(ii)]{Korbas:Osaka}). So in order to prove that $\w_4x\neq0$ if $x\in H^r(\gr_{2^t,4})\setminus\{0\}$ and $r<2^t-4$, it suffices to verify that the ideal $(g_{2^t-2},g_{2^t-1},g_{2^t})$ contains no polynomial divisible by $w_4$ whose cohomological degree is less than $2^t$. However, this is straightforward from the fact that neither $g_{2^t-2}$ nor $g_{2^t-1}$ is divisible by $w_4$. This fact follows from (\ref{eq_explicitly_g}). Namely, the monomial $w_2^{2^{t-1}-1}$ appears in $g_{2^t-2}$ with nonzero coefficient, and $w_2^{2^{t-1}-2}w_3$ appears in $g_{2^t-1}$ with nonzero coefficient.
\end{proof}

 For a fixed integer $r\geqslant0$ the image of the map $sp^*:H^r(\gr_{2^t,4})\rightarrow H^r(W_{3,1}^{2^t})$ is equal to $F_r^r$. If a class in $H^r(\gr_{2^t,4})$ is decomposable, i.e., if it is a polynomial in classes of strictly smaller degree, then applying $sp^*$ we get a polynomial in classes from $F_i^i$ for $0<i<r$. Consequently, $sp^*$ induces an epimorphism
\begin{equation}\label{eq:izo-nerastavljivi}
H^r(\gr_{2^t,4})/D(H^r(\gr_{2^t,4}))\longrightarrow F_r^r/D(F_r^r),
\end{equation}
where $D(H^r(\gr_{2^t,4}))$ is the image of the cup product map 
\[\bigoplus_{i=1}^{r-1}H^i(\gr_{2^t,4})\otimes H^{r-i}(\gr_{2^t,4})\xrightarrow{\smile}H^r(\gr_{2^t,4}),\]
and $D(F_r^r)$ is the image of 
$\displaystyle\bigoplus_{i=1}^{r-1}F_i^i\otimes F_{r-i}^{r-i}\xrightarrow{\smile}F_r^r.$
\begin{lemma}\label{izo-nerastavljivi}
    The epimorphism (\ref{eq:izo-nerastavljivi}) is an isomorphism if $r>4$.
\end{lemma}
\begin{proof}
    The proof is the same as the proof of \cite[Lemma 3.5]{BasuChakraborty}. If $x\in H^r(\gr_{2^t,4})$ is such that $sp^*(x)\in D(F_r^r)$, we need to prove that $x\in D(H^r(\gr_{2^t,4}))$, i.e., that $x$ is decomposable. We have $sp^*(x)=\sum_ja_jb_j$, where $a_j$ and $b_j$ are some classes in the groups $F_i^i$ for $0<i<r$, and so
    \[sp^*(x)=\sum_jsp^*(\alpha_j)sp^*(\beta_j)=sp^*\bigg(\sum_j\alpha_j\beta_j\bigg),\]
    for some classes $\alpha_j$ and $\beta_j$ from the groups $H^i(\gr_{2^t,4})$ for $0<i<r$. In other words, $sp^*(x)=sp^*(y)$ for some $y\in D(H^r(\gr_{2^t,4}))$. Now, $x-y\in\ker sp^*=\w_4H^{r-4}(\gr_{2^t,4})$ (see (\ref{kersp^*})). Therefore, there is a class $z\in H^{r-4}(\gr_{2^t,4})$ such that $x=y+\w_4z$, and $y+\w_4z\in D(H^r(\gr_{2^t,4}))$ because $r-4>0$.
\end{proof}

According to \cite[Theorem 2.1(2)]{Korbas:Osaka} the smallest $r$ such that $H^r(\gr_{2^t,4})\setminus\im p^*\neq\emptyset$ is $2^t-4$. Let us choose a class
\begin{equation}\label{indec}
\wa_{2^t-4}\in H^{2^t-4}(\gr_{2^t,4})\setminus\im p^*.
\end{equation}
Since $2^t-4$ is the smallest cohomological degree with such a class, $\wa_{2^t-4}$ is indecomposable, i.e., its coset in $H^{2^t-4}(\gr_{2^t,4})/D(H^{2^t-4}(\gr_{2^t,4}))$ is nonzero.

\begin{theorem}\label{generators}
     The Stiefel--Whitney classes $\w_2$, $\w_3$, $\w_4$, along with the class $\wa_{2^t-4}$, generate the algebra $H^*(\widetilde G_{2^t,4})$.
\end{theorem}
\begin{proof}
The statement of the theorem is true for $t=3$ by \cite[Theorem 3.3]{Rusin}.

Suppose now that $t>3$. We need to prove that for every $r\geqslant0$ all classes in $H^r(\gr_{2^t,4})$ are polynomials in $\w_2$, $\w_3$, $\w_4$ and $\wa_{2^t-4}$. 

If $r<2^t-4$, then $H^r(\gr_{2^t,4})\subseteq\im p^*$, and since $\w_2$, $\w_3$ and $\w_4$ generate $\im p^*$, we are done.

\medskip

Let $r=2^t-4$. We know that the coset of $\wa_{2^t-4}$ in $H^{2^t-4}(\gr_{2^t,4})/D(H^{2^t-4}(\gr_{2^t,4}))$ is nonzero, and since $D(H^{2^t-4}(\gr_{2^t,4}))\subseteq\im p^*$, it suffices to prove that it is the only nonzero coset, i.e., $H^{2^t-4}(\gr_{2^t,4})/D(H^{2^t-4}(\gr_{2^t,4}))\cong\mathbb Z_2$. The assumption $t>3$ implies $2^t-4>4$, and so due to Lemma \ref{izo-nerastavljivi} it is enough to verify that
\begin{equation}\label{r=2^t-4}
F_{2^t-4}^{2^t-4}/D(F_{2^t-4}^{2^t-4})\cong\mathbb Z_2.
\end{equation}
By Proposition \ref{cohW} and equation (\ref{F_r^r}) we have 
\begin{equation}\label{D}
\ww_2\in F_2^2,\, \ww_3\in F_3^3 \,\mbox{ and }\, c_{2^t-4}\in F_{2^t-4}^{2^t-4},
\end{equation}
and the elements of $D(F_{2^t-4}^{2^t-4})$ are exactly the polynomials in $\ww_2$ and $\ww_3$ (of cohomological degree $2^t-4$). Since we know that the quotient $F_{2^t-4}^{2^t-4}/D(F_{2^t-4}^{2^t-4})$ is nontrivial, Proposition \ref{cohW} implies that the coset of $c_{2^t-4}$ must be nonzero, and that it is the only nonzero coset. This proves (\ref{r=2^t-4}) and concludes the case $r=2^t-4$.

\medskip

For $r>2^t-4$ we need to show that every class in $H^r(\gr_{2^t,4})$ is decomposable, and by Lemma \ref{izo-nerastavljivi} this is equivalent to $F_r^r=D(F_r^r)$. So let $x\in F_r^r$, $x\neq0$. If we prove that $x$ is a polynomial in $\ww_2$, $\ww_3$ and $c_{2^t-4}$, then (\ref{D}) (along with the fact that the filtration is stable) ensures that $x\in D(F_r^r)$.

\begin{figure}[h]
    \centering
    \includegraphics[width=\linewidth]{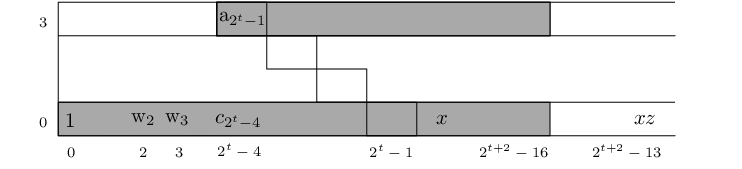}
    \caption{$E_\infty$-page of the spectral sequence for $S^3\rightarrow W_{3,1}^{2^t} \xrightarrow{sp} \gr_{2^t,4}$}
    \label{fig:sn}
\end{figure}

Assume to the contrary that $x$ is not expressible as a polynomial in $\ww_2$, $\ww_3$ and $c_{2^t-4}$. By Proposition \ref{cohW} and equality (\ref{a^2=0}) this means that $x=\mathrm{a}_{2^t-1}y$, where $y$ is a polynomial in $\ww_2$, $\ww_3$ and $c_{2^t-4}$ (in expressing $x$ we can ignore the monomials that do not contain $\mathrm{a}_{2^t-1}$, because they belong to $D(F_r^r)$). By the Poincar\'e duality there exists a class $z\in H^{2^{t+2}-13-r}(W_{3,1}^{2^t})$ with the property $xz\neq0$ in $H^{2^{t+2}-13}(W_{3,1}^{2^t})$ (recall that $W_{3,1}^{2^t}$ is a closed connected manifold of dimension $2^{t+2}-13$). Since $x=\mathrm{a}_{2^t-1}y$ and $\mathrm{a}_{2^t-1}^2=0$, the class $z$ can be chosen to be a polynomial in $\ww_2$, $\ww_3$ and $c_{2^t-4}$ only. The stability of the filtration and (\ref{D}) now imply $z\in F_{2^{t+2}-13-r}^{2^{t+2}-13-r}$, and since $x\in F_r^r$ we have
\[xz\in F_{2^{t+2}-13}^{2^{t+2}-13}\cong E_\infty^{2^{t+2}-13,0}=0,\]
since the base space $\gr_{2^t,4}$ is a manifold of dimension $2^{t+2}-16$ (see Figure \ref{fig:sn}). However, this contradicts the fact $xz\neq0$ and finishes the proof. 
\end{proof}

As we have already established in (\ref{kersp^*}), $\ker sp^*=\w_4 H^*(\gr_{2^t,4})$, so it is obvious from Theorem \ref{generators} and Proposition \ref{cohW} that 
\begin{equation}\label{w_i}
    sp^*(\w_2)=\ww_2 \,\mbox{ and }\, sp^*(\w_3)=\ww_3.
\end{equation}
Also, $sp^*(\wa_{2^t-4})=c_{2^t-4}+f(\ww_2,\ww_3)$ for some polynomial $f$ in two variables. Namely, otherwise, by Proposition \ref{cohW} and (\ref{w_i}) we would have $sp^*(\wa_{2^t-4})=g(\ww_2,\ww_3)=sp^*(g(\w_2,\w_3))$ for some polynomial $g$, which would imply $\wa_{2^t-4}-g(\w_2,\w_3)\in\ker sp^*=\w_4 H^{2^t-8}(\gr_{2^t,4})$, and this contradicts the fact $\wa_{2^t-4}\notin\im p^*$. 

If we alter $\wa_{2^t-4}$ by adding $f(\w_2,\w_3)$, then we achieve that
\begin{equation}\label{a}
    sp^*(\wa_{2^t-4})=c_{2^t-4},
\end{equation}
while both (\ref{indec}) and Theorem \ref{generators} still hold.

\begin{remark}
    As we have already mentioned, the class $\wa_{2^t-4}$ is indecomposable (it cannot be written as a polynomial in classes of strictly smaller degree), and it is easily seen from (\ref{imp*}) that $\w_2$, $\w_3$ and $\w_4$ are indecomposable as well. Theorem \ref{generators} asserts that these four are the only indecomposable classes in $H^*(\gr_{2^t,4})$ (up to addition of decomposable ones).
\end{remark}

\section{Cohomology algebra $H^*(\widetilde G_{2^t,4})$}
\label{sec:main_proof}

In this section we prove Theorem \ref{Cohomology_G_2^t,4}. As before, $t\geqslant3$ is a fixed integer.

\begin{proposition}\label{a_na_kvadrat}
    If $\wa_{2^t-4}\in H^{2^t-4}(\gr_{2^t,4})$ is the indecomposable class (chosen in the previous section), then in $H^*(\gr_{2^t,4})$ one has
    \[\widetilde a_{2^t-4}^2 = \widetilde P\wa_{2^t-4} + \widetilde Q,\]
    for some polynomials $\widetilde P$ and $\widetilde Q$ in $\widetilde w_2$, $\widetilde w_3$ and $\widetilde w_4$.
\end{proposition}
\begin{proof}
    Consider once again the fiber bundle $S^{2^t-4}\xrightarrow{i} W_{3,1}^{2^t} \xrightarrow{p_1} \widetilde G_{2^t,3}$. Using the notation from the previous section, we have that $1=i^*(1)$ and $s=i^*(c_{2^t-4})$ constitute a basis for $H^*(S^{2^t-4})$. The Leray--Hirsch theorem (see e.g.\ \cite[Theorem 4D.1.]{Hatcher}) provides us an additive isomorphism $\Phi: H^*(\gr_{2^t,3})\otimes H^*(S^{2^t-4}) \rightarrow H^*(W_{3,1}^{2^t})$ given by $\Phi(x\otimes1) = p_1^*(x)$ and $\Phi(x\otimes s) = p_1^*(x)\cdot c_{2^t-4}$. 
    
    By \cite[Corollary 3.6.1]{CP}, $H^*(\gr_{2^t,3})$ has an additive basis consisting of elements $\w_2^b\w_3^c$ and $\wa_{2^t-1} \w_2^b \w_3^c$, where $(b,c)\in T$, and $T$ is the set of all pairs of nonnegative integers $(b,c)$ such that for all $i\in\{0,1,\dots,t-1\}$ one has either $b<2^{t-1} - 2^i$ or  $c<2^i-1$. Let $N$ be the cardinality of $T$, so that this basis of $H^*(\gr_{2^t,3})$ has $2N$ elements. The isomorphism $\Phi$ then gives us an additive basis of $H^*(W_{3,1}^{2^t})$ consisting of elements (see (\ref{defw})):
    \begin{itemize}
        \item $\Phi(\w_2^b\w_3^c\otimes1)=p_1^*(\w_2^b\w_3^c)=\ww_2^b\ww_3^c$;
        \item $\Phi(\wa_{2^t-1}\w_2^b\w_3^c\otimes1)=p_1^*(\wa_{2^t-1}\w_2^b\w_3^c)=\ww_2^b\ww_3^c\mathrm a_{2^t-1}$;
        \item $\Phi(\w_2^b\w_3^c\otimes s)=p_1^*(\w_2^b\w_3^c)\cdot c_{2^t-4}=\ww_2^b\ww_3^cc_{2^t-4}$;
        \item $\Phi(\wa_{2^t-1}\w_2^b\w_3^c\otimes s)=p_1^*(\wa_{2^t-1}\w_2^b\w_3^c)\cdot c_{2^t-4}=\ww_2^b\ww_3^cc_{2^t-4}\mathrm a_{2^t-1}$;
    \end{itemize}
    for $(b,c)\in T$. Note that $\dim H^*(W_{3,1}^{2^t})=4N$.

    Consider now the sphere bundle $S^3\rightarrow W_{3,1}^{2^t} \xrightarrow{sp} \gr_{2^t,4}$. From Lemma \ref{1/2dim} we have
    \[\dim(\im sp^*) = \frac{1}{2} \dim H^*(W_{3,1}^{2^t}) = 2N.\]
    On the other hand, by (\ref{w_i}) and (\ref{a}) we know that $2N$ (linearly independent) elements of the first and the third kind above are in $\im sp^*$, and so they make up the whole basis of $\im sp^*$. Thus, since $c_{2^t-4}^2\in \im sp^*$, we can represent it via the basis elements as
    \[c_{2^t-4}^2 = c_{2^t-4} \overline P + \overline Q,\]
    where $\overline P$ and $\overline Q$ are polynomials in $\mathrm w_2$ and $\mathrm w_3$ only. Now, again by (\ref{w_i}) and (\ref{a})
    \[sp^*\bigg(\wa_{2^t-4}^2 - \big(\wa_{2^t-4} \overline P(\w_2,\w_3) + \overline Q(\w_2,\w_3)\big)\bigg) = c_{2^t-4}^2 - (c_{2^t-4} \overline P + \overline Q) = 0,\]
    and from (\ref{kersp^*}) we conclude
    \[\wa_{2^t-4}^2 = \wa_{2^t-4}\overline P(\w_2,\w_3) + \overline Q(\w_2,\w_3) + \widetilde w_4 \overline R,\]
    for some $\overline R\in H^{2^{t+1}-12}(\gr_{2^t,4})$, which, by Theorem \ref{generators}, must be a polynomial in $\w_2$, $\w_3$, $\w_4$ and $\wa_{2^t-4}$. When we regroup summands, we get 
    \[\wa_{2^t-4}^2 =\widetilde P\wa_{2^t-4} +\widetilde Q,\]
    for some polynomials $\widetilde P$ and $\widetilde Q$ in $\w_2$, $\w_3$ and $\w_4$.
    \end{proof}

If we choose polynomials $\widetilde P$ and $\widetilde Q$ provided by Proposition \ref{a_na_kvadrat}, then in the polynomial algebra $\mathbb Z_2[w_2,w_3,w_4,a_{2^t-4}]$ we can observe the corresponding polynomials $P$ and $Q$ in variables $w_2$, $w_3$ and $w_4$. According to Theorem \ref{generators}, the algebra morphism
\[\Psi:\mathbb Z_2[w_2,w_3,w_4,a_{2^t-4}]\rightarrow H^*(\gr_{2^t,4}),\]
given by $\Psi(w_2)=\w_2$, $\Psi(w_3)=\w_3$, $\Psi(w_4)=\w_4$ and $\Psi(a_{2^t-4})=\wa_{2^t-4}$, is an epimorphism. Therefore, the induced map
\begin{equation}\label{Psiisom}
\overline\Psi:\mathbb Z_2[w_2,w_3,w_4,a_{2^t-4}]/\ker\Psi\rightarrow H^*(\gr_{2^t,4})
\end{equation}
is an isomorphism, and in order to prove Theorem \ref{Cohomology_G_2^t,4} it now suffices to show 
\[\ker\Psi=(g_{2^t-2},g_{2^t-1},g_{2^t},a_{2^t-4}^2+Pa_{2^t-4}+Q).\]
Due to Proposition \ref{a_na_kvadrat} we know that $a_{2^t-4}^2+Pa_{2^t-4}+Q\in\ker\Psi$.
Also, according to (\ref{imp*}), for the map $p^*$ induced by $p:\gr_{2^t,4}\rightarrow G_{2^t,4}$ we have 
\[\im p^*\cong\mathbb Z_2[w_2,w_3,w_4]/(g_{2^t-3},g_{2^t-2},g_{2^t-1},g_{2^t})=\mathbb Z_2[w_2,w_3,w_4]/(g_{2^t-2},g_{2^t-1},g_{2^t}),\] 
since $g_{2^t-3}=0$ (see \cite[Lemma 2.3(ii)]{Korbas:Osaka}). So $g_{2^t-2},g_{2^t-1},g_{2^t}\in\ker\Psi$, and we conclude that, if $I:=(g_{2^t-2},g_{2^t-1},g_{2^t},a_{2^t-4}^2+Pa_{2^t-4}+Q)$, then
\begin{equation}\label{I_podskup_ker}
   I\subseteq\ker\Psi.
\end{equation}
We are left to prove the reverse containment. In order to do so, in the following subsection we detect a Gr\"obner basis for the ideal $I$.

\subsection{Gr\"obner basis for $I$}

For $k=4$ the equation (\ref{eq_rec_g}) simplifies to
\begin{equation}\label{eq_rec_g_k=4}
 g_r =   w_2 g_{r-2} + w_3 g_{r-3} +   w_4 g_{r-4}, \quad r\geqslant 1
\end{equation}
(recall that we have defined $g_{-1}=g_{-2}=g_{-3}=0$),
and its generalization (\ref{eq_rec_g_gen}) reduces to
\begin{equation}\label{eq_rec_g_k=4_gen}
    g_r =   w_2^{2^i} g_{r-2^{i+1}} + w_3^{2^i} g_{r-3\cdot 2^i} +   w_4^{2^i} g_{r-2^{i+2}}, \quad r\geqslant2^{i+2}-3,\, i\geqslant0.
\end{equation}
Since $I=(g_{2^t-2},g_{2^t-1},g_{2^t},a_{2^t-4}^2+Pa_{2^t-4}+Q)$ and $g_{2^t-3}=0$, (\ref{eq_rec_g_k=4}) implies $g_r\in I$ for all $r\geqslant2^t-2$. In particular, $I=(G)$, where
\begin{equation}\label{Grbaza}
G=\{g_{2^t-3+2^i}\mid 0\leqslant i\leqslant t-1\}\cup\{g_{2^t},a_{2^t-4}^2+Pa_{2^t-4}+Q\}.     
\end{equation}
We are going to prove that this basis $G$ of the ideal $I\unlhd\mathbb Z_2[w_2,w_3,w_4,a_{2^t-4}]$ is its Gr\"obner basis. The following two lemmas will be helpful.

\begin{lemma}\label{posl}
    For all integers $r\geqslant-3$ and $i\geqslant0$,
    \[w_3^{2^i-1}g_r^{2^i}=g_{2^i(r+3)-3}.\]
\end{lemma}
\begin{proof}
    The proof is by induction on $i$. For $i=0$ the equality is simply $g_r=g_r$. For $i=1$ we need to verify
    \begin{equation}\label{i=1}
      w_3g_r^2=g_{2r+3}  \quad\mbox{ for all } r\geqslant-3. 
    \end{equation}
    For $r\in\{-3,-2,-1\}$ both sides of this equation are zero, and for $r=0$ (\ref{i=1}) reduces to $w_3=g_3$, which is easily checked (e.g.\ from (\ref{eq_explicitly_g})). If $r\geqslant1$ and if we assume $w_3g_j^2=g_{2j+3}$ for $-3\leqslant j<r$, then by (\ref{eq_rec_g_k=4}) and (\ref{eq_rec_g_k=4_gen})
    \begin{align*}
    w_3g_r^2&=w_3(w_2 g_{r-2} + w_3 g_{r-3} +w_4 g_{r-4})^2=w_2^2w_3g_{r-2}^2 + w_3^2w_3g_{r-3}^2 +w_4^2w_3g_{r-4}^2\\
            &=w_2^2g_{2r-1} + w_3^2g_{2r-3} +w_4^2g_{2r-5}=g_{2r+3}.
    \end{align*}
    This induction on $r$ proves (\ref{i=1}) and concludes the case $i=1$.
    
    Suppose now that $i\geqslant2$ and that
    \[w_3^{2^{i-1}-1}g_r^{2^{i-1}}=g_{2^{i-1}(r+3)-3} \quad\mbox{ for all } r\geqslant-3.\]
    If we square this equation, and then multiply with $w_3$ we obtain
    \[w_3^{2^i-1}g_r^{2^i}=w_3g_{2^{i-1}(r+3)-3}^2=g_{2^i(r+3)-3}\]
    by (\ref{i=1}). This completes the induction on $i$ and the proof of the lemma.
\end{proof}

\begin{lemma}\label{g3=g4}
    For $0\leqslant i\leqslant t-1$ none of the monomials in $g_{2^t-3+2^i}$ contains the variable $w_4$.
\end{lemma}
\begin{proof}
From Lemma \ref{posl} we have
\[g_{2^t-3+2^i}=g_{2^i(2^{t-i}-2+3)-3}=w_3^{2^i-1}g_{2^{t-i}-2}^{2^i},\]
so it is enough to prove that polynomials $g_{2^m-2}$ ($m\geqslant0$) do not contain monomials divisible by $w_4$. This obviously holds for $m\in\{0,1\}$, which verifies the base for our induction on $m$. Proceeding to the induction step, suppose that $m\geqslant2$ and that $w_4$ does not appear in $g_{2^j-2}$ for $0\leqslant j<m$. Since by (\ref{eq_rec_g_k=4_gen})
\begin{align*}
  g_{2^m-2}&=w_2^{2^{m-2}}g_{2^{m-1}-2}+w_3^{2^{m-2}}g_{2^{m-2}-2}+w_4^{2^{m-2}}g_{-2}\\
           &=w_2^{2^{m-2}}g_{2^{m-1}-2}+w_3^{2^{m-2}}g_{2^{m-2}-2}, 
\end{align*}
the variable $w_4$ does not appear in $g_{2^m-2}$ either. 
\end{proof}

\begin{theorem}\label{thmGrebner}
  The set $G$ (see (\ref{Grbaza})) is a Gr\"obner basis for $I\unlhd\mathbb Z_2[w_2,w_3,w_4,a_{2^t-4}]$ w.r.t.\ the lexicographic monomial ordering $\preccurlyeq$ in which $w_3 \prec w_2\prec  w_4 \prec a_{2^{t}-4}$.  
\end{theorem}
\begin{proof}
    We know that $G$ generates $I$, and according to Theorem \ref{buchberger} it suffices to show that $S$-polynomials of any two elements of $G$ reduce to zero modulo $G$. Since $S(f,f)=0$ we can consider only the $S$-polynomials of two distinct elements of $G$.

    It is obvious that $\mathrm{LM}(a_{2^t-4}^2+Pa_{2^t-4}+Q)=a_{2^t-4}^2$ (the polynomials $P$ and $Q$ do not contain $a_{2^t-4}$), and it is straightforward from (\ref{eq_explicitly_g}) that $\mathrm{LM}(g_{2^t})=w_4^{2^{t-2}}$. Also, for $0\leqslant i\leqslant t-1$, $\mathrm{LM}(g_{2^t-3+2^i})$ does not contain neither $a_{2^t-4}$ nor $w_4$ (by Lemma \ref{g3=g4}).
    
    So, if at least one polynomial in the $S$-polynomial is from $\{g_{2^t},a_{2^t-4}^2+Pa_{2^t-4}+Q\}$, then the leading monomials of the two polynomials are coprime, and \cite[Theorem 5.66]{Becker} ensures that the $S$-polynomial reduces to zero modulo $G$. We are therefore left to prove that for two distinct $i,j\in\{0,1,\ldots,t-1\}$, $S(g_{2^t-3+2^i},g_{2^t-3+2^j})$ reduces to zero.

    It is well-known (and easily seen from (\ref{eq_explicitly_g})) that the modulo $w_4$ reduction of $g_r$, $r\geqslant0$, is the corresponding polynomial $g_r$ for $k=3$. Lemma \ref{g3=g4} then implies that $g_{2^t-3+2^i}\in\mathbb Z_2[w_2,w_3,w_4]$ for $k=4$ equals the polynomial $g_{2^t-3+2^i}\in\mathbb Z_2[w_2,w_3]$ for $k=3$. On the other hand, it was shown in \cite[Theorem 3.5]{CP} that the polynomials $g_{2^t-3+2^i}$, $0\leqslant i\leqslant t-1$, constitute a Gr\"obner basis for the ideal $(g_{2^t-2},g_{2^t-1})\unlhd\mathbb Z_2[w_2,w_3]$ w.r.t.\ the lexicographic monomial ordering with $w_3<w_2$. Therefore, $S(g_{2^t-3+2^i},g_{2^t-3+2^j})$ does reduce to zero for all $i,j\in\{0,1,\ldots,t-1\}$.
\end{proof}

Moreover, it was proved in \cite[Proposition 3.4]{CP} that 
\[\mathrm{LM}(g_{2^t-3+2^i})=w_2^{2^{t-1}-2^i}w_3^{2^i-1}, \quad 0\leqslant i\leqslant t-1.\]
So, a monomial $w_2^bw_3^c$ is not divisible by any of these leading monomials if and only if for every $i\in\{0,1,\ldots,t-1\}$ one has either $b<2^{t-1}-2^i$ or $c<2^i-1$, i.e, if and only if $(b,c)\in T$ (see the proof of Proposition \ref{a_na_kvadrat}). Along with the facts $\mathrm{LM}(g_{2^t})=w_4^{2^{t-2}}$ and $\mathrm{LM}(a_{2^t-4}^2+Pa_{2^t-4}+Q)=a_{2^t-4}^2$, by our definition of Gr\"obner bases, this establishes the following corollary.

\begin{corollary}\label{cor:aditivna_baza1}
    If
$B=\Big\{a_{2^t-4}^rw_4^dw_2^bw_3^c\mid r<2, d<2^{t-2}, (b,c)\in T\Big\}$,
then the cosets of elements of $B$ form an additive basis of the quotient $\mathbb Z_2[w_2,w_3,w_4,a_{2^t-4}]/I$.

Consequently, if $N$ is the cardinality of the set $T$ (as in the proof of Proposition \ref{a_na_kvadrat}), then 
\[\dim\Big(\mathbb Z_2[w_2,w_3,w_4,a_{2^t-4}]/I\Big)=2^{t-1}\cdot N.\]
\end{corollary}

Before moving on to the proof of Theorem \ref{Cohomology_G_2^t,4} let us note one more thing. In the same way as above, one can verify that the set $\{g_{2^t-3+2^i}\mid 0\leqslant i\leqslant t-1\}\cup\{g_{2^t}\}$ is a Gr\"obner basis for the ideal
$(g_{2^t-2},g_{2^t-1},g_{2^t})\unlhd\mathbb Z_2[w_2,w_3,w_4]$
(w.r.t.\ the lexicographic ordering with $w_3<w_2<w_4$), and consequently, that the cosets of monomials $w_4^dw_2^bw_3^c$ such that $d<2^{t-2}$ and $(b,c)\in T$, constitute a vector space basis of the quotient $\mathbb Z_2[w_2,w_3,w_4]/(g_{2^t-2},g_{2^t-1},g_{2^t})$. However, we know that this quotient is isomorphic to $\im p^*$, where $p:\gr_{2^t,4}\rightarrow G_{2^t,4}$ is the two-fold covering. Therefore, $ \dim(\im p^*)=2^{t-2}\cdot N$, and Lemma \ref{1/2dim} applies to give us
\begin{equation}\label{dimH^*}
    \dim H^*(\gr_{2^t,4})=2^{t-1}\cdot N.
\end{equation}

\subsection{Proof of Theorem \ref{Cohomology_G_2^t,4}}

Recall that we are left to prove the equality $I=\ker\Psi$. By (\ref{I_podskup_ker}) we know that $I\subseteq\ker\Psi$, and so we have the following short exact sequence of vector spaces:
\[0 \rightarrow\ker\Psi/I \rightarrow \mathbb Z_2[w_2,w_3,w_4,a_{2^t-4}]/I \rightarrow \mathbb Z_2[w_2,w_3,w_4,a_{2^t-4}]/\ker\Psi\rightarrow 0.\]
According to Corollary \ref{cor:aditivna_baza1},
\[\dim\Big(\mathbb Z_2[w_2,w_3,w_4,a_{2^t-4}]/I\Big)=2^{t-1}\cdot N,\]
while the isomorphism (\ref{Psiisom}) and equality (\ref{dimH^*}) imply
\[\dim\Big(\mathbb Z_2[w_2,w_3,w_4,a_{2^t-4}]/\ker\Psi\Big)=\dim H^*(\gr_{2^t,4})=2^{t-1}\cdot N,\]
leading to the conclusion that $\dim(\ker\Psi/I)=0$, i.e., $I=\ker\Psi$. This completes the proof of our main theorem.

\subsection{Some additional facts about $H^*(\gr_{2^t,4})$}

First of all, since we now have $I=\ker\Psi$, Corollary \ref{cor:aditivna_baza1} and the isomorphism (\ref{Psiisom}) give us the following assertion.
\begin{corollary}\label{cor:aditivna_baza}
    The set of cohomology classes
\[\widetilde B\!=\!\!\Big\{\wa_{2^t-4}^r\w_4^d\w_2^b\w_3^c\mid\! r<2, d<2^{t-2}\!, \big(\forall i\!\in\!\{0,\!1,\ldots,t-1\}\!\big) \, \,b<2^{t-1}-2^i \,\,\,\vee\,\,\, c<2^i-1\!\Big\}\]
is a vector space basis of $H^*(\gr_{2^t,4})$.
\end{corollary}
\begin{remark}\label{remark1}
 We know that the basis elements with $r=0$ belong to $\im p^*$, and since $\dim H^*(\gr_{2^t,4})=2\dim(\im p^*)$ (by Lemma \ref{1/2dim}), those with $r=1$ do not belong to $\im p^*$.
\end{remark}

Now, in order to prove that $\wa_{2^t-4}^2\neq0$ in $H^*(\gr_{2^t,4})$, we recall from \cite{JP} that the square of an (arbitrarily chosen) indecomposable class from $H^{2^t-4}(\gr_{2^t-1,3})$ is nonzero. So it suffices to prove that $\widetilde j^*(\wa_{2^t-4})\in H^{2^t-4}(\gr_{2^t-1,3})$ is indecomposable, where $\widetilde j:\gr_{2^t-1,3}\hookrightarrow\gr_{2^t,4}$ is induced by the inclusion $\mathbb R^{2^t-1}\hookrightarrow\mathbb R^{2^t}$ (namely, then $\widetilde j^*(\wa_{2^t-4}^2)=\widetilde j^*(\wa_{2^t-4})^2\neq0$, and so $\wa_{2^t-4}^2\neq0$). We prove this by using \cite[Lemma 2.1(b)]{JP}, and the indecomposability of $\widetilde j^*(\wa_{2^t-4})$ follows immediately from the next lemma and the fact that $\wa_{2^t-4}$ is itself indecomposable, i.e., $\wa_{2^t-4}\notin\im p^*$ (see (\ref{indec})). Note that the next lemma deals with the "unoriented" Grassmannians.
\begin{lemma}
    Let $w_1:H^{2^t-4}(G_{2^t,4})\rightarrow H^{2^t-3}(G_{2^t,4})$ be multiplication with the Stiefel--Whitney class $w_1=w_1(\gamma_{2^t,4})\in H^1(G_{2^t,4})$, and $j^*:H^{2^t-4}(G_{2^t,4})\rightarrow H^{2^t-4}(G_{2^t-1,3})$ the map induced by $j:G_{2^t-1,3}\hookrightarrow G_{2^t,4}$ (which again comes from the inclusion $\mathbb R^{2^t-1}\hookrightarrow\mathbb R^{2^t}$). Then
    \[\ker w_1 \cap \ker j^* = 0.\]
\end{lemma}




    \begin{proof}
    By (\ref{Borel}), $H^*(G_{2^t,4})\cong \mathbb Z_2[w_1,w_2,w_3,w_4]/(\overline w_{2^t-3}, \overline w_{2^t-2}, \overline w_{2^t-1}, \overline w_{2^t})$. Let $f\in \mathbb Z_2[w_1,w_2,w_3,w_4]$ be a homogeneous polynomial of (cohomological) degree $2^t-4$ such that the corresponding class $[f]\in H^{2^t-4}(G_{2^t,4})$ is in $\ker w_1 \cap \ker j^*$. From $[f]\in \ker w_1$, in $\mathbb Z_2[w_1,w_2,w_3,w_4]$ we have 
            \begin{equation}\label{eq1}
            w_1 f = \alpha \overline w_{2^t-3},
            \end{equation}
        for some $\alpha\in\mathbb Z_2$. On the other hand, the map $j^*:H^*(G_{2^t,4})\rightarrow H^*(G_{2^t-1,3})$ is the reduction modulo $w_4$ (see \cite[Lemma 2.2(b)]{JP}).
        So, from $[f]\in \ker j^*$, we get $0=j^*[f] = [\rho(f)]$, where $\rho :\mathbb Z_2[w_1,w_2,w_3,w_4] \rightarrow \mathbb Z_2[w_1,w_2,w_3]$ is the mod $w_4$ map. Since there are no relations among polynomials of degree $2^t-4$ (namely, by (\ref{Borel}), $H^*(G_{2^t-1,3})\cong \mathbb Z_2[w_1,w_2,w_3]/(\overline w_{2^t-3}, \overline w_{2^t-2}, \overline w_{2^t-1})$), we conclude $\rho(f) = 0$, so 
        \begin{equation}\label{eq2}
            f = w_4 h,
        \end{equation}
        for some $h\in\mathbb Z_2[w_1,w_2,w_3,w_4]$.

        Now from (\ref{eq1}) and (\ref{eq2})
        \begin{equation}\label{eq3}
            w_1 w_4 h = \alpha \overline w_{2^t-3} \quad \mbox{ in } \mathbb Z_2[w_1,w_2,w_3,w_4].
            \end{equation}
        From (\ref{eq_explicitly_w}), we see that $\overline w_{2^t-3}$ has $w_1^{2^t-3}$ as a summand, while there is no such summand on the left-hand side in (\ref{eq3}). We conclude $
        \alpha = 0$ and thus $f = 0$.
    \end{proof}

    So, we have shown that $\wa_{2^t-4}^2\neq0$, because $\widetilde j^*(\wa_{2^t-4})^2\neq0$. Furthermore, according to \cite[(3.14)]{JP}, $\widetilde j^*(\wa_{2^t-4})^2=\widetilde g_{2^t-4}\cdot j^*(\wa_{2^t-4})+\gamma\w_2^{2^t-4}$, for some (unknown) $\gamma\in\mathbb Z_2$ and a unique nonzero $\widetilde g_{2^t-4}\in H^{2^t-4}(\gr_{2^t-1,3})$ (which corresponds to the coset of the polynomial $g_{2^t-4}$ via (\ref{imp*}) for $k=3$ and $n=2^t-1$) . On the other hand, by Proposition \ref{a_na_kvadrat},
    \[\widetilde j^*(\wa_{2^t-4})^2=\widetilde j^*(\widetilde P)\cdot\widetilde j^*(\wa_{2^t-4})+\widetilde j^*(\widetilde Q),\]
    and we conclude that $\widetilde P$ must be nonzero.

    Therefore, we have proved the following proposition.

    \begin{proposition}\label{Pneq0}
      For the indecomposable class $\wa_{2^t-4}\in H^{2^t-4}(\gr_{2^t,4})$ we have $\wa_{2^t-4}^2\neq0$. Moreover, in the description of $H^*(\gr_{2^t,4})$ given in Theorem \ref{Cohomology_G_2^t,4} the polynomial $P$ cannot be zero.
    \end{proposition}

    \section{The case $t=3$}
    \label{sec:example}
    
    In this section we restrict our attention to the case $t=3$, i.e., to the cohomology algebra $H^*(\widetilde G_{8,4})$. Theorem \ref{Cohomology_G_2^t,4} in this case has the following form:
     \[H^*(\widetilde G_{8,4}) \cong \frac{\mathbb Z_2[ w_2, w_3, w_4, a_4]}{(g_6, g_7, g_8, a_4^2+P a_4+Q)}.\]
    We will use the additive basis from Corollary \ref{cor:aditivna_baza} to get more information on this algebra. Namely, we will compute the polynomials $P$ and $Q$ up to a coefficient from $\mathbb Z_2$. 
    
    \begin{proposition}\label{prop:G8,4}
    We have the ismorphism of graded algebras:
        \[H^*(\widetilde G_{8,4}) \cong \frac{\mathbb Z_2[ w_2, w_3, w_4,  a_4]}{(g_6, g_7, g_8, a_4^2 + a_4 w_2^2 +\gamma w_2w_3^2)},\]
        for some $\gamma\in\mathbb Z_2$.
    \end{proposition}

    \begin{proof}
    It suffices to prove that in $H^*(\widetilde G_{8,4})$ one has
    \begin{equation}
        \label{a^2--t=3}
        \wa_4^2=\wa_4\w_2^2 +\gamma\w_2\w_3^2,
    \end{equation}
    for some $\gamma\in\mathbb Z_2$. Corollary \ref{cor:aditivna_baza} gives us the additive basis $\widetilde B$ of $H^*(\widetilde G_{8,4})$. For $0\leqslant j\leqslant16$ let $\widetilde B_j$ be the subset of $\widetilde B$ consisting of elements with cohomological degree $j$. Then $\widetilde B_j$ is a basis of $H^j(\widetilde G_{8,4})$. We have:
 \begin{align*}
\widetilde B_5&=\{\w_2 \w_3\},\, \widetilde B_6=\{\wa_4 \w_2, \w_2 \w_4, \w_3^2\},\, \widetilde B_8=\{\wa_4 \w_4, \wa_4 \w_2^2, \w_2\w_3^2,\w_2^2 \w_4\},\\
\widetilde B_{10}&=\{\wa_4\w_2\w_4,\wa_4\w_3^2,\w_3^2\w_4\},\, \widetilde B_{16}=\{\wa_4\w_2\w_3^2\w_4\}.
 \end{align*}
 Also, it is routine to calculate: $g_6=w_2^3+w_3^2$ and $g_7=w_2^2w_3$; and so
 \begin{equation}\label{g_6g_7}
 \w_2^3+\w_3^2=0,\, \w_2^2\w_3=0, \mbox{ and consequently, } \w_3^3=\w_2^3\w_3=0,\, \w_2^4=\w_2\w_3^2.
 \end{equation}
 
 We can write $\widetilde a_4^2\in H^8(\widetilde G_{8,4})$ in the basis $\widetilde B_8$ as
        \[\wa_4^2 = \alpha\wa_4 \w_4+ \beta\wa_4 \w_2^2+ \gamma\w_2\w_3^2 + \delta\w_2^2 \w_4,\]
        where $\alpha, \beta,\gamma, \delta\in\mathbb Z_2$. If $\widetilde j:\gr_{7,3}\hookrightarrow\gr_{8,4}$ is the embedding, it is known that $\widetilde j^*:H^*(\gr_{8,4})\rightarrow H^*(\gr_{7,3})$ is the reduction modulo $\w_4$ (see \cite[(2.6)]{JP}). Therefore,
        \[\widetilde j^*(\wa_4)^2 = \widetilde j^*(\wa_4^2) = \beta\widetilde j^*(\wa_4)\w_2^2+ \gamma\w_2\w_3^2.\]
        On the other hand, we already know that $\widetilde j^*(\wa_4)$ is an indecomposable class in $H^4(\gr_{7,3})$, and so by \cite[(3.14) and (3.12)]{JP} we have
        \[\widetilde j^*(\wa_4)^2=\widetilde j^*(\wa_4)\widetilde g_4+ \overline\gamma\w_2\w_3^2=\widetilde j^*(\wa_4)\w_2^2+ \overline\gamma\w_2\w_3^2,\]
        for some $\overline\gamma\in\mathbb Z_2$ (it is straightforward to check $g_4=w_2^2$ for $k=3$; see e.g.\ \cite[Example 1]{CP}). Since the elements $\widetilde j^*(\wa_4)\w_2^2,\w_2\w_3^2\in H^8(\gr_{7,3})$ are linearly independent (by \cite[(3.9)]{JP}) it follows from the previous two equations that $\beta=1$ and $\gamma=\overline\gamma$. We now have
        \begin{equation}\label{beta=1}
        \wa_4^2 = \alpha\wa_4 \w_4+\wa_4 \w_2^2+\gamma\w_2\w_3^2 + \delta\w_2^2 \w_4.
        \end{equation}
        Therefore, in order to obtain (\ref{a^2--t=3}) we are left to prove $\alpha=\delta=0$.

        We shall do that by applying appropriate Steenrod squares on (\ref{beta=1}). By Wu's formula (see \cite[p.\ 91]{MilnorSt}), we have 
        \[Sq^1(\widetilde w_2) = \widetilde w_3,\,\, Sq^1(\widetilde w_3) = 0,\,\, Sq^1(\widetilde w_4) = 0,
        \]
        \[ Sq^2(\widetilde w_2) = \widetilde w_2^2,\,\, Sq^2(\widetilde w_3) = \widetilde w_2\widetilde w_3,\,\, Sq^2(\widetilde w_4) = \widetilde w_2 \widetilde w_4.\]
        In what follows, we implicitly use the Cartan formula to calculate the action of Steenrod squares on polynomials in $\w_2$, $\w_3$, and $\w_4$.

        Consider first $Sq^4: H^{12}(\widetilde G_{8,4}) \rightarrow H^{16}(\widetilde G_{8,4})$. Since the dimension of $\widetilde G_{8,4}$ is exactly 16, this $Sq^4$ is the multiplication with the Wu class $v_4 \in H^4(\gr_{8,4})$. After an easy calculation, from \cite[Theorem 11.14]{MilnorSt} and \cite[Theorem 1.1]{BartikKorbas} one gets $v_4 = 0$. Thus, $Sq^4(\wa_4 \w_2 \w_3^2)=0$. On the other hand, by using (\ref{g_6g_7}) and Cartan's formula it is routine to verify that $Sq^i(\w_2 \w_3^2)=0$ for $1\leqslant i\leqslant4$, and so according to (\ref{beta=1}) and (\ref{g_6g_7}):
        \begin{align*}
            0 &= Sq^4(\wa_4 \w_2 \w_3^2)=\sum_{i=0}^4Sq^{4-i}(\wa_4)Sq^i(\w_2 \w_3^2)=Sq^4(\wa_4)\w_2 \w_3^2=\wa_4^2\w_2 \w_3^2\\
            &= (\alpha\wa_4 \w_4+\wa_4 \w_2^2+\gamma\w_2\w_3^2 + \delta\w_2^2 \w_4)\w_2 \w_3^2=\alpha\wa_4\w_2\w_3^2\w_4.
        \end{align*}
       Since $\widetilde B_{16}=\{\wa_4\w_2\w_3^2\w_4\}$, we conclude $\alpha = 0$. Therefore, (\ref{beta=1}) now becomes
       \begin{equation}\label{alfa=0}
        \wa_4^2 = \wa_4 \w_2^2+\gamma\w_2\w_3^2 + \delta\w_2^2 \w_4.
        \end{equation}

        Since $\widetilde B_5=\{\w_2\w_3\}$, if $Sq^1(\wa_4)$ were nonzero, we would have $Sq^1(\wa_4) = \w_2 \w_3$. But then the identity $Sq^1Sq^1=0$ would imply
        \[0=Sq^1(Sq^1(\wa_4)) = Sq^1(\w_2 \w_3) = \w_3^2,\]
        which is false, because $\w_3^2\in\widetilde B_6$. Therefore, $Sq^1(\wa_4)=0$.

       Consequently, $Sq^2(\wa_4^2)=(Sq^1(\wa_4))^2=0$. On the other hand, from (\ref{alfa=0}) we get
        \begin{align*}
            0 &= Sq^2(\wa_4^2)= Sq^2(\wa_4 \w_2^2+\gamma\w_2\w_3^2 + \delta\w_2^2 \w_4)\\
            &=Sq^2(\wa_4)\w_2^2+\wa_4\w_3^2+\gamma\w_2^2\w_3^2+\delta\big(\w_3^2 \w_4+\w_2^3 \w_4\big)= Sq^2(\wa_4)\w_2^2 + \wa_4\w_3^2,
        \end{align*}
        where we used (\ref{g_6g_7}) for the last equality. Since $\wa_4 \w_3^2\in\widetilde B_{10}$, from Remark \ref{remark1} we have $\wa_4 \w_3^2\notin\im p^*$, and we conclude $Sq^2(\wa_4) \notin \im p^* $ (where $p: \widetilde G_{8,4} \rightarrow G_{8,4}$ is the double covering map). So, when we present this element in the basis $\widetilde B_6$, we must have $Sq^2(\wa_4) = \wa_4 \w_2 + \lambda\w_2 \w_4 + \mu\w_3^2$, where $\lambda, \mu \in \mathbb Z_2$. 
        
        Let us show that $\lambda=\mu=0$. The Adem relation $Sq^2 Sq^2 = Sq^3 Sq^1$ gives us $Sq^2(Sq^2(\wa_4)) = Sq^3(Sq^1(\wa_4)) = 0$. On the other hand,
        \begin{align*}
            0&=Sq^2(Sq^2(\wa_4))=Sq^2(\wa_4 \w_2 + \lambda\w_2 \w_4 + \mu\w_3^2)=Sq^2(\wa_4)\w_2+\wa_4\w_2^2\\
             &=(\wa_4 \w_2 + \lambda\w_2 \w_4 + \mu\w_3^2)\w_2+\wa_4\w_2^2=\lambda\w_2^2\w_4 + \mu\w_2\w_3^2,
        \end{align*}
        and since both $\w_2^2\w_4$ and $\w_2\w_3^2$ are basis elements from $\widetilde B_8$, we have $\lambda=\mu=0$. Therefore, we now know that 
        \begin{equation}\label{Sq2}
         Sq^2(\wa_4) = \wa_4 \w_2.   
        \end{equation}

        Finally, $Sq^4(\wa_4^3)=v_4\wa_4^3=0$, and using (\ref{alfa=0}), (\ref{Sq2}) and (\ref{g_6g_7}) we calculate:
        \begin{align*}
            0&=Sq^4(\wa_4^3)=Sq^4(\wa_4\cdot\wa_4^2)=Sq^4(\wa_4)\wa_4^2+\wa_4Sq^4(\wa_4^2)=\wa_4^4+\wa_4\big(Sq^2(\wa_4)\big)^2\\
             &=(\wa_4 \w_2^2+\gamma\w_2\w_3^2 + \delta\w_2^2 \w_4)^2+\wa_4^3\w_2^2=\wa_4^2 \w_2^4+\delta\w_2^4 \w_4^2+\wa_4^2\wa_4\w_2^2\\
             &=\wa_4^2 \w_2^4+(\wa_4 \w_2^2+\gamma\w_2\w_3^2 + \delta\w_2^2 \w_4)\wa_4\w_2^2=\delta\wa_4\w_2^4 \w_4=\delta\wa_4\w_2\w_3^2\w_4
        \end{align*}
       (the fact $\w_2^4 \w_4^2=0$ follows from $\w_2^4 \w_4^2\in H^{16}(\gr_{2^t,4})\cap\im p^*=0$, using $\widetilde B_{16}=\{\wa_4\w_2\w_3^2\w_4\}$ and Remark \ref{remark1}). As before, we conclude $\delta=0$, and (\ref{alfa=0}) simplifies to (\ref{a^2--t=3}), finishing the proof.
\end{proof}

Our description of $H^*(\gr_{2^t,4})$ from Theorem \ref{Cohomology_G_2^t,4} is not complete, since we did not compute the polynomials $P$ and $Q$. We were only able to prove that $P$ must be nonzero (Proposition \ref{Pneq0}). Even in this smallest case ($t=3$) the coefficient $\gamma$ remained undetermined. As we have seen in the proof of Proposition \ref{prop:G8,4}, this coefficient equals the (undetermined) coefficient $\gamma$ from \cite[Theorem 1.1(a)]{JP} for $t=3$. Therefore, the task of computing the polynomials $P$ and $Q$ in general seems quite challenging.


\bibliographystyle{amsplain}

\begin{thebibliography}{10}





\bibitem{BartikKorbas}
{V.\ Bart\' ik and J.\ Korba\v s}, {\it Stiefel--Whitney characteristic classes and parallelizability of Grassmann manifolds}, Rend. Circ. Mat. Palermo (2)\ {\bf 33} (1984) 19--29.

%
\bibitem{BasuChakraborty}
{S.\ Basu and P.\ Chakraborty}, {\it On the cohomology ring and upper characteristic rank of Grassmannian
of oriented $3$-planes}, J.\ Homotopy Relat.\ Struct.\ {\bf 15} (2020) 27--60.


%
\bibitem{Becker}
{T.\ Becker and V.\ Weispfenning}, {\it Gr\"obner Bases: A Computational Approach to Commutative
Algebra}, Graduate Texts in Mathematics, Springer-Verlag, New York (1993).


%
\bibitem{Borel}
{ A.\ Borel}, {\it La cohomologie mod 2 de certains espaces homog\`enes}, Comm.\ Math.\ Helv.\ {\bf 27} (1953) 165--197.

%



%
\bibitem{CP}
{U.\ A.\ Colovi\'c and B.\ I.\ Prvulovi\'c,} {\it Gr\"obner bases in the mod $2$ cohomology of oriented Grassmann manifolds $\widetilde G_{2^t,3}$}, Math. Slovaca {\bf 74(1)} (2024) 195–208.

%

%
\bibitem{Hatcher}
A.\ Hatcher, \textit{Algebraic Topology}, Cambridge University Press, Cambridge (2002).

%
\bibitem{JP}
{M.\ Jovanovi\'c and B.\ I.\ Prvulovi\'c,} {\it On the mod 2 cohomology algebra of oriented Grassmannians}, {J. Homotopy Relat. Struct.} {\bf 19} (2024) 379--396.



%
\bibitem{Korbas:Osaka}
{J.\ Korba\v s,} {\it The characteristic rank and cup-length in oriented Grassmann manifolds}, {Osaka J. Math.} {\bf 52} (2015) 1163--1172.

%
\bibitem{KorbasRusin:Palermo}
{J.\ Korba\v s and T.\ Rusin,} {\it A note on the $\mathbb Z_2$-cohomology algebra of oriented Grassmann manifolds}, {Rend.\ Circ.\ Mat.\ Palermo, II.\ Ser} {\bf 65} (2016) 507--517.



\bibitem{Wendt}
{\' A.\ Matszangosz and M.\ Wendt,} {\it The mod 2 cohomology rings of oriented Grassmannians via Koszul complexes},  {Math.\ Z.} {\bf 308:2} (2024).

\bibitem{McCleary}
{J.\ McCleary,} {\it A User's Guide to Spectral Sequences}, Second Edition, Cambridge University Press, Cambridge (2001).


%
\bibitem{MilnorSt}
J.\ W.\ Milnor, J.\ D.\ Stasheff, \textit{Characteristic Classes}, Ann.\ of Math.\ Studies \textbf{76}, Princeton University Press, New Jersey (1974).


 \bibitem{bane-marko}{B. I. Prvulovi\' c, M. Radovanovi\' c,} {\it On the characteristic rank of vector bundles over oriented Grassmannians}, {Fund. Math.} \textbf{244} (2019), 167--190.


\bibitem{Rusin}
T.\ Rusin, \textit{A note on the cohomology ring of the oriented Grassmann manifolds $\widetilde G_{n,4}$}, Arch. Math. (Brno), \textbf{55}(5):319–331, (2019).



\end{thebibliography}

\end{document}